\documentclass[notitlepage,reqno,11pt]{amsart}
\usepackage{latexsym,amssymb, epsfig, amsmath,amsfonts, subfigure,amsthm, tikz}

\usepackage{rotating}
\usepackage[toc,page]{appendix}
\usepackage{color}
\usepackage{multirow}
\usepackage{relsize}
\usepackage{microtype}

\usepackage[margin=1in]{geometry}

\numberwithin{equation}{section}
 
\newtheorem{lemma}{Lemma}[section]
\newtheorem{theorem}{Theorem}[section]

\newtheorem{remark}{Remark}[section]

\newlength{\defbaselineskip}
\newcommand{\setlinespacing}[1]%
           {\setlength{\baselineskip}{#1 \defbaselineskip}}
\newcommand{\doublespacing}{\setlength{\baselineskip}%
                           {1.5 \defbaselineskip}}

\newcommand{\CC}{{\mathbb C}}

\newcommand{\RR}{{\mathbb R}}

\newcommand{\deq}{\stackrel{\rm d}{=}}
\newcommand{\beql}[1]{\begin{equation}\label{#1}}
\newcommand{\eeq}{\end{equation}}

\newcommand{\beqal}[1]{\begin{eqnarray}\label{#1}}
\newcommand{\eeqa}{\end{eqnarray}}
\newcommand{\beq}{\begin{displaymath}}
\newcommand{\eeqno}{\end{displaymath}}
\newcommand{\bali}[1]{\begin{align}\label{#1}}
\newcommand{\eali}{\begin{align}}
\newcommand{\balino}{\begin{align*}}
\newcommand{\ealino}{\begin{align*}}

\newcommand{\ep}{\epsilon}

\newcommand{\Var}{\text{\rm Var}}
\newcommand{\Cov}{\text{\rm Cov}}

\newcommand{\qandq}{\quad\mbox{and}\quad}

\newcommand{\qforq}{\quad\mbox{for}\quad}

\newcommand{\qasq}{\quad\mbox{as}\quad}

\newcommand{\non}{\nonumber}

\newcommand{\baa}{\begin{eqnarray*}}
\newcommand{\eaa}{\end{eqnarray*}}

\newcommand{\ttl}{\Large Path Properties of a Generalized Fractional Brownian Motion}

\begin{document}

\title[]{\ttl}

\allowdisplaybreaks

\author{Tomoyuki Ichiba}
\address{Department of Statistics \& Applied Probability, University of California, Santa Barbara, CA 93106}
\email{ichiba@pstat.ucsb.edu}

\author{Guodong Pang}
\address{The Harold and Inge Marcus Department of Industrial and
Manufacturing Engineering,
College of Engineering,
Pennsylvania State University,
University Park, PA 16802}
\email{gup3@psu.edu}

\author{Murad S. Taqqu}
\address{Department of Mathematics and Statistics, Boston University, Boston, MA 02215}
\email{murad@math.bu.edu}

\date{\today}

\begin{abstract} 
\doublespacing
The generalized fractional Brownian motion is a Gaussian self-similar process whose increments are not necessarily stationary. It appears in applications  as the scaling limit of a  shot noise process with a power law shape function and non-stationary noises with a power-law variance function. 
 In this paper we study sample path properties of the generalized fractional Brownian motion, including H{\"o}lder continuity, path differentiability/non-differentiability, and functional and local Law of the Iterated Logarithms. 

\end{abstract}


\maketitle


\doublespacing

\allowdisplaybreaks

\section{Introduction}

We consider the generalized fractional Brownian motion (GFBM) $X:=\{X(t): t\in \RR_+\}$ defined via the following (time-domain) integral representation: 
\begin{equation} \label{def-X}
\{X(t)\}_{t\in \RR} \deq \left\{  c \int_\RR \left( (t-u)_+^{\alpha} - (-u)_+^{\alpha} \right)    |u|^{-\gamma/2} B( {\mathrm d} u)  \right\}_{t\in \RR},
\end{equation}
where
\begin{equation} \label{def-X-g-c}
 \gamma \in [0,1), \quad
\alpha \in \Big(-\frac{1}{2}+ \frac{\gamma}{2}, \ \frac{1}{2}+ \frac{\gamma}{2}\Big),
\end{equation}
and $c= c(\alpha, \gamma) \in \RR_+$ is the normalization constant.  
Here, $B( {\mathrm d} u)$ is a Gaussian random measure on $\RR$ with the Lebesgue control measure $ du$.
It is shown in Proposition 5.1 \cite{pang-taqqu} that the process $X$ is a continuous mean-zero Gaussian process with $X(0)=0$, and has the self-similarity property  
with Hurst parameter 
\begin{equation} \label{eq: Hurst para}
H=\alpha - \frac{\gamma}{2} + \frac{1}{2} \in (0,1).
\end{equation}

This process arises as the scaling limit of the so-called power-law non-stationary shot noise processes which have
the shot shape function of power-law with parameter $\alpha$ and the non-stationary noise distributions with a power-law variance function of parameter $\gamma$. This is established in \cite{pang-taqqu}. With i.i.d. (stationary) noises, the scaled power-law shot noise processes converge to the standard FBM, see, e.g., \cite[Chapter 3.4]{PT17} and \cite{kluppelberg2004fractional}. Note that the power-law in the shot shape function captures the long range dependence while the power-law in the non-stationary noises captures the dispersions of their variabilities, and thus does not contribute to the long range dependence. 

The GFBM $X$ in \eqref{def-X} is a natural generalization of the standard FBM, since it preserves the same long range dependence structure as FBM, while the power-law perturbation of the  (Brownian) Gaussian random measure not only introduces non-stationarity (in the increments) but also preserves the important self-similarity property. 

There are three parameters, $H$, $\alpha$  and $\gamma$ and two relations \eqref{def-X-g-c}-\eqref{eq: Hurst para}. 
 Eliminating $\alpha$ yields the following representation of the self-similar process X with Hurst parameter $H \in (0,1)$ and (scale/shift) parameter  $\gamma \in (0, 1) $: 
\begin{equation} \label{def-X-alternative}
\{X(t)\}_{t\in \RR} \deq \left\{  c \int_\RR \left( (t-u)_+^{H- \frac{1}{2} +\frac{ \gamma}{2} } - (-u)_+^{H- \frac{1}{2} +\frac{ \gamma}{2} } \right)    |u|^{-\gamma/2} B( {\mathrm d} u)  \right\}_{t\in \RR}.
\end{equation}
Evidently when $\gamma=0$, this becomes the standard FBM $B^H$:
\begin{equation} \label{FBM-rep}
\{B^H(t)\}_{t\in \RR} \deq \left\{  c \int_\RR \left( (t-u)_+^{H-\frac{1}{2}} - (-u)_+^{H-\frac{1}{2}} \right)   B( {\mathrm d} u)  \right\}_{t\in \RR}. 
\end{equation}

Although one may think of $  |u|^{-\gamma/2} $ as a time change of the Brownian motion 
 which introduces non-stationarity increments,  we observe from the representation in \eqref{def-X-alternative} that  for a given Hurst parameter value $H \in (0,1)$, the parameter $\gamma$ also shifts the exponents
in $ (t-u)_+^{H-\frac{1}{2}} - (-u)_+^{H-\frac{1}{2}}$ by the positive amount $\frac{\gamma}{2} \in (0,\frac{1}{2})$.  For instance, for $H=\frac{1}{4}$, the exponent in the FBM $B^H$ is $H-\frac{1}{2}=-\frac{1}{4}$, but with $\gamma = \frac{3}{4}$, that exponent in the process $X$ becomes $H- \frac{1}{2} +\frac{ \gamma}{2} =\frac{1}{8}$. For another instance, for   $H=\frac{3}{4}$, the exponent in the FBM $B^H$ is $H-\frac{1}{2}=\frac{1}{4}$, but with $\gamma = \frac{3}{4}$, the exponent in the process $X$ becomes $H- \frac{1}{2} +\frac{ \gamma}{2} =\frac{5}{8}$.
 We see that the positive shift in the exponent makes the function $(t-u)_+^{H- \frac{1}{2} +\frac{ \gamma}{2} } - (-u)_+^{H- \frac{1}{2} +\frac{ \gamma}{2} }$ smoother than $ (t-u)_+^{H-\frac{1}{2}} - (-u)_+^{H-\frac{1}{2}}$. On the other hand, the function $  |u|^{-\gamma/2} $ has the opposite effect, making the paths ``rougher".  It is then interesting to ask how the  parameter $\gamma$ affects the path properties of the GFBM.

 To answer this question, we focus on the  sample path properties of the GFBM $X$, H{\"o}lder continuity,  path differentiability/non-differentiability, and functional and local Law of Iterated Logarithms (FLIL and LLIL, respectively). 
In Theorem \ref{thm-Holder}, we prove that the paths of the process $X$ are H{\"o}lder continuous with parameter $H-\epsilon$ for $\epsilon>0$.  
In Theorems \ref{thm-FLIL}, \ref{thm-LLIL} and \ref{thm-LLIL-composition}, we prove the functional  Law of Iterated Logarithm (FLIL) and  LLIL as well as an LLIL for the composition of the GFBM $X$ with itself, which again, depends only on the Hurst parameter $H$. 
These are somewhat surprising results, indicating that the nice path properties of H{\"o}lder continuity, FLIL and LLIL  are preserved by the construction of the GFBM $X$ in \eqref{def-X} and \eqref{def-X-alternative},  { and are not being affected by the parameter $\gamma$}.

On the other hand, the differentiability of the paths of the GFBM $X$ is affected by the parameter $\gamma$.  It is well known that the FBM $B^H$ is non-differentiable for $H\in (0,1)$. 
In Theorem \ref{thm-nondiff}, we show that if the parameters $(\alpha,\gamma)$ are in the region $\{ \alpha \in (1/2,1/2+\gamma/2), \gamma \in (0,1)\}$, leading to $H\in (1/2,1)$, the paths of $X$ are differentiable, while in the rest of parameter ranges, the paths of $X$ are non-differentiable. It is interesting to observe that for $H\in (1/2,1)$, there are distinct path differentiability properties in the two regions distinguished by $\alpha>1/2$ (differentiable) and $\alpha\le 1/2$ (non-differentiable). In addition, we show that when $\alpha>1/2$, the paths of the GFBM $X$ is once continuously differentiable but not twice (with probability one), and we derive the first-order derivative.  These results are distinct from the non-differentiability property of the FBM $B^H$.

It is worth mentioning that all these properties of the standard FBM $B^H$ rely critically on the stationary increment property, i.e., the familiar elegant covariance function and the second moment of its increment (see \eqref{eq: FBM cov.f} and \eqref{eq: FBM 2ndM}). The proofs of these properties are relatively straightforward, and have become standard textbook materials \cite{PT17}.  However, for the GFBM $X$ in \eqref{def-X}, non-stationary increments result in a rather complicated covariance function (see \eqref{eqn-Psi}). 
For the proof of the H{\"o}lder continuity, we provide a useful decomposition of the increment of the GFBM $X$, and then evaluate their increments separately. This decomposition may turn out to be useful in other purposes. 
For the other properties, we draw upon some important results that were established for general Gaussian processes (some with self-similarity properties), for example, the (non)differentiability property by Yeh  
\cite{yeh1967differentiability}, FLIL by Taqqu \cite{taqqu1977law}, and  local LIL and compositions of certain Gaussian processes with itself by Arcones \cite{arcones1995law}.  For the GFBM $X$ in \eqref{def-X}, due to its non-stationary increment property and the particular structure of its covariance function, it is challenging to verify some of the technical conditions imposed in these results.  The proofs of the non-differentiability and FLIL rely critically upon the H{\"o}lder continuity property we establish.

We also remark that the GFBM  $X$ in \eqref{def-X} is different from the so-called Brownian semi-stationary  (BSS)  processes introduced by Barndoff-Nielsen and Schmiegel  \cite{barndorff2009brownian}, which was used to study volatility/intermittency inference problems in financial markets. The process was introduced to circumvent the non-semimartingale issues on the inference problems concerning the underlying volatility process based on realized quadratic variation (see the multi-power variation for BSS process in \cite{barndorff2011multipower}). However, their assumptions on the spot intermittency process exclude functions of the type $|u|^{-\gamma/2}$ as we assume (see, for example, equation (4.7) in \cite{barndorff2011multipower}). 

FBMs have been recently used to study ``rough" volatility \cite{gatheral2018volatility,bayer2016pricing,livieri2018rough}. On the other hand, non-stationary increments have been well recognized in various financial data, see, e.g., \cite{bassler2007nonstationary,mccauley2008martingales}.  The GFBM $X$ in \eqref{def-X} and the path properties studied in this paper may be useful in the study of ``rough" volatility. 

We start in the next section with some preliminary results on basic properties of the GFBM $X$. The H{\"o}lder continuity, differentiability/non-differentiability, FLIL and LLIL results are stated and proved in Sections \ref{sec-Holder}, \ref{sec-nondiff}, \ref{sec-FLIL} and \ref{sec-LLIL}, respectively.

\section{Some preliminaries}

A striking distinction from the standard FBM is the non-stationary increment property. 
Recall that  the standard FBM $B^H$ with the Hurst index $H$ has the covariance function: for $s, t\in \RR_+$, 
 \begin{equation} \label{eq: FBM cov.f}
 E\bigl[B^H(s) B^H(t)\bigr] = \frac{1}{2}c^2 (t^{2H} + s^{2H} - |t-s|^{2H}),
 \end{equation} 
 and the second moment of its increment:
 \begin{equation} \label{eq: FBM 2ndM}
 E\bigl[(B^H(s) - B^H(t))^2\bigr] = c^2 |t-s|^{2H}. 
 \end{equation}
 This stationary increment property plays the fundamental role in proving many properties of FBM and the associated processes, for example, stochastic integrals with respect to FBM. 

For the GFBM $X$ in (\ref{def-X}) the covariance function $\Psi  $ between $X(s) $ and $X(t)$ and the second moment function $\Phi $ of its increment $X(s) - X(t) $ are given, respectively,  by 
\begin{align} \label{eqn-Psi}
\Psi(s,t) & := \Cov(X(s), X(t)) = E[X(s)X(t)] \non\\
&= c^2 \int_\RR \Big(  \left( (t-u)_+^{\alpha} - (-u)_+^{\alpha} \right)  \left( (s-u)_+^{\alpha} - (-u)_+^{\alpha} \right) \Big)  |u|^{-\gamma} {\rm d}u, \non\\
&= c^2 \int_0^s (t-u)^{\alpha} (s-u)^{\alpha} u^{-\gamma} {\rm d}u 
+ c^2 \int_0^{\infty} ((t+u)^{\alpha} - u^{\alpha}) ((s+u)^{\alpha} - u^{\alpha})u^{-\gamma}{\rm d}u, 
\end{align}
\begin{align} \label{eqn-Phi}
\Phi(s,t)&:=E\bigl[(X(s) - X(t))^2\bigr]  = c^2 \int_\RR \Big(  (t-u)_+^{\alpha} - (s-u)_+^{\alpha}  \Big)^2  |u|^{-\gamma} {\rm d}u \non\\
&= c^2  \int_s^t (t-u)^{2\alpha} u^{-\gamma} {\rm d}u 
+ c^2\int_0^s ( (t-u)^{\alpha} - (s-u)^{\alpha} )^2 u^{-\gamma} {\rm d}u \non\\
&\qquad  + c^{2}\int_0^\infty ((t+u)^{\alpha} -(s+u)^\alpha)^2 u^{-\gamma} {\rm d} u, \quad 0\le s \le t.  
\end{align}

When $\gamma =0$, the GFBM $X$ in (\ref{def-X}) becomes the standard FBM with the covariance function (\ref{eq: FBM cov.f}) and stationary second moments (\ref{eq: FBM 2ndM}) of increments.  

For standard FBM $B^H$, we usually distinguish two cases: $H \in (0, \frac{1}{2} )$ and $H \in ( \frac{1}{2}, 1)$, which corresponds to the exponents in \eqref{FBM-rep} being negative and positive, respectively. 
However, for the GFBM $X$, we distinguish the following two cases: 
$$
H \in \left(0, \frac{1-\gamma}{2} \right) \qandq H \in \left(\frac{1-\gamma}{2},1 \right), \qforq \gamma \in [0,1),
$$
which correspond to the exponents in \eqref{def-X-alternative} being negative or positive. Note that $\gamma$ can be very close 1, in which case { the interval $ \big(\frac{1-\gamma}{2},1 \big) $ in the second scenario 
 becomes very close to $(0,1)$, the whole range of the Hurst parameter $H$. }
The two cases can be also written in terms of $\alpha$ and $\gamma$: 
$$
\alpha \in \left(-\frac{1-\gamma}{2}, 0 \right) \qandq \alpha \in \left(0,  \frac{1+\gamma}{2} \right),  \qforq \gamma \in [0,1). 
$$

\begin{remark}[The role of $\gamma$] \label{rem: 2.2}
{\em 
We highlight the following on the role of the parameter $\gamma$:
\begin{enumerate}
\item[(i)] When $\gamma \in (0, 1)$, the increment is not second-order stationary,  {that is, $\Phi(s,t)$ is not a function of $|s-t|$.} 
\item[(ii)]  $\Var(c^{-1}X(t)) = c^{-2} \Psi(t,t) = t^{2H}$ is decreasing in $\gamma$ and increasing $\alpha$ (where $c = c(\alpha, \gamma) \in \RR_+$ is the normalization parameter in Lemma \ref{lem-const}). 
\item[(iii)] Flexibility for Hurst parameter $H$: Figure~\ref{fig-alpha-gamma} illustrates the range of $\alpha$ and $\gamma$ for the Hurst parameter $H \in (0,1)$. The middle dotted line corresponds to the value $H=0.5$ (including the special cases $(\alpha=0, \gamma =0)$ and $(\alpha \approx 0.5, \gamma \approx 1)$), instead of the single value $H=0.5$ in the case of BM; see further discussions in Remark~\ref{rem-H=0.5}. 
For $\alpha\approx 0$ and $\gamma\approx 1$, the Hurst parameter $H$ can be arbitrarily close to zero, while for $\alpha \approx -0.5$ and $\gamma \approx 0$ (which is close to the FBM case), the same is also true.
\item[(iv)] Roughness of paths: H{\"o}lder continuity, FLIL and LLIL hold with the Hurst parameter $H$ as for the standard FBM $B^H$. However, the GFBM $X$ is differentiable when $\alpha \in (1/2, 1/2+\gamma/2)$ and $\gamma \in (0,1)$ (resulting $H \in (1/2,1)$), while it is non-differentiable when  $\alpha \in (-1/2+\gamma/2, 1/2]$ and $\gamma \in (0,1)$ (resulting $H\in (0,1)$).
 \end{enumerate}} 
\end{remark}

\begin{center}
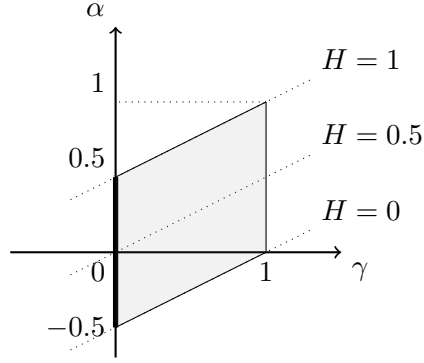
\begin{figure} \label{fig-alpha-gamma}
\begin{tikzpicture}[scale=2]
\draw[thick,->] (-0.7, 0.0) -- (1.5, 0.0) node[anchor=north west] {$\gamma$}; 
\draw[thick,->] (0,-0.7) -- (0, 1.5) node[anchor=south east] {$\alpha$}; 
\draw[] (0,0) node[anchor = north east] {$0$};
\draw[] (0,-0.5) -- (1, 0) node[below] {$1$};
\draw[] (1, 1) -- (0,0.5) node[anchor = south east]{0.5};
\draw[] (0, -0.5) node[anchor=east] {$-0.5$}; 
\draw[] (1,0) -- (1,1) node[]{};
\draw[dotted] (1,1) -- (0, 1) node[anchor = south east] {$1$};
\draw[dotted] (-0.3,0.35) -- (1.3, 1.15) node[above right] {$H = 1 $}; 
\draw[dotted] (-0.3,-0.65) -- (1.3, 0.15) node[above right] {$H = 0 $}; 
\draw[dotted] (-0.3,-.15) -- (1.3, 0.65) node[above right] {$H = 0.5 $}; 
\draw[line width = 0.7mm] (0, -0.5) -- (0, 0.5); 
\draw[fill=gray, opacity=0.1] ((0,0.5) -- (1,1) -- (1,0) -- (0,-0.5) -- (0,0.5) -- cycle; 
\end{tikzpicture}
\caption{The set of parameters $\,(\gamma, \alpha)\,$ given in (\ref{def-X-g-c}) is shown in the shaded area. The boundary points are not included in  (\ref{def-X-g-c}). The dotted lines corresponding to the Hurst parameters $\,H = 0, 0.5, 1\,$ are plotted, respectively.  The thick line segment corresponds to the FBM with $H \in (0, 1) $ and $\gamma = 0 $. In the neighborhood	of the point $\,\gamma = 1\,$, $\, \alpha \, =\,  0 \,$, the Hurst index is close to $\,0\,$.}
\end{figure}
\end{center}

\begin{remark} \label{rem-H=0.5}
{\em 
{
Although the FBM $B^H$ becomes a standard BM when $H=1/2$, this is not the case for the GFBM $X$.
The values of $\alpha$ and $\gamma$ corresponding to the case $H=1/2$ lie in the line $\alpha- \frac{\gamma}{2} =0$ for $\gamma \in [0,1)$. The GFBM $X$ only becomes a standard BM in the special case $\gamma =0$. This is   due to the fact that the process $X$ does not have stationary increments if $\gamma>0$. 
For $\gamma  \in (0,1)$ and $\alpha =  \frac{\gamma}{2}$, the process $X$ provides an example of a $H$-self-similar Gaussian process with Hurst parameter $H=1/2$ which is not a  BM. 
We remark that there are $H$-self-similar processes  with $H=1/2$, that may not be Gaussian, see \cite{bai2014generalized}. It is clear that when $H=1/2$ and $ 2 \alpha = \gamma > 0 $,  the process $X$  is not a martingale with respect to the filtration $\mathcal F^{B}(s) := { \sigma} \{B(u), u \le s \} $, $s \in \mathbb R$ generated by the BM $B$,  because for every $s < t$, 
\begin{equation*}
\begin{split}
E [ X(t) - X(s) \, \vert \, \mathcal F^{B}(s) ] \, &=\, E \Big[ \int_{\mathbb R}  ((t-u)^{\alpha}_{+} - (s-u)^{\alpha}_{+} ) \lvert u \rvert^{-\alpha} B({\mathrm d}u) \, \Big \vert \, \mathcal F^{B}(s)  \Big]  \\
&= \, \int^{s}_{-\infty} ((t-u)^{\alpha} - (s-u)^{\alpha} ) \lvert u \rvert^{-\alpha} B({\mathrm d}u) \neq 0 . 
\end{split}
\end{equation*}
It is worth mentioning the work on ``fake" Brownian motions  constructed from martingales in \cite{fan2015mimicking,hobson2016mimicking}.
}
}
\end{remark}

\begin{remark}
{\em
In  \cite[Sections 5.1 and 5.2]{pang-taqqu}, generalized FBMs are stated in a more general form with the additional terms involving $(t-u)_{-}^\alpha - (-u)_{-}^\alpha$ in the integrands. This can be treated similarly with additional terms, and so we focus on the representations of $X$ in \eqref{def-X}. } 
\end{remark}

\subsection{The normalization constant $c=c(\alpha,\gamma)$}
{
With the increments  
\begin{equation} \label{eqn-wt-B}
 \widetilde{B}(t) \, :=\,  B(- t) - B(0) \,, \quad t \ge 0,
\end{equation}
 independent of the increments $\, B(t) - B(0)  \,$, $\,t \ge 0 \,$,} we obtain 
\begin{equation} \label{X-rep-1}
\begin{split}
c^{-1} X(t) \, =&\,  \int_{\mathbb R} \big( ( t-u)_{+}^{\alpha} - ( - u )_{+}^{\alpha} \big) \lvert u\rvert^{-\gamma/2 } B({\mathrm d} u ) \\
\, =&\,  \int^{t}_{0} ( t-u)^{\alpha} \lvert u \rvert^{-\gamma/2} B({\mathrm d} u) + \int^{\infty}_{0} [  ( t+v)^{\alpha} v^{-\gamma/2}- v^{\alpha-\gamma/2}] \widetilde{B}({\mathrm d} v) \\
\, \stackrel{(d)}{=}&\, \bigg( \int^{1}_{0} ( 1-v)^{\alpha} v^{-\gamma/2} B({\mathrm d} v) + \int^{\infty}_{0} [  ( 1+v)^{\alpha} v^{-\gamma/2} - v^{\alpha-\gamma/2}] \widetilde{B}({\mathrm d} v) \bigg) \cdot t^{\alpha - \frac{\,\gamma\,}{\,2\,} + \frac{\,1\,}{\,2\,}} \, , 
\end{split}
\end{equation}
where the last equality is a distributional identity from the scaling property of Brownian motion, and $ \stackrel{(d)}{=}$ denotes ``equal in distribution". Thus, one can express the constant $c$ in terms of Beta and Gamma functions as follows.

\begin{lemma} \label{lem-const}
With  \begin{equation} \label{eqn-c-def}
c= c(\alpha, \gamma):=\kappa(\alpha, \gamma)^{-1/2},
\end{equation} 
and
\begin{equation} \label{eq: kappa1}
 \kappa(\alpha, \gamma) \, :=\,  \text{\rm Beta} ( 1 - \gamma, 2 \alpha + 1 ) +  \Big( \frac{\,\Gamma ( 1 - \gamma) \,}{\,\Gamma (-2\alpha) \,}  - \frac{\,2 \Gamma ( 1+ \alpha  - \gamma) \,}{\,\Gamma (-\alpha) \,} \Big) \Gamma ( - 1 - 2\alpha + \gamma),  
\end{equation}
the GFBM $X(t)$ in \eqref{def-X} is then normalized:
\begin{equation} \label{eq:VarXt}
\Var(X(t)) = t^{2\alpha-\gamma+1} = t^{2H}, \quad t\ge 0.
\end{equation} 
\end{lemma}

\begin{proof}
By the distributional identity \eqref{X-rep-1} and the It\^o isometry of stochastic integral, we obtain 
\begin{equation} \label{eq: varcX}
\begin{split}
&\text{Var} ( c^{-1} X(t)) \, =\, \,  \mathbb E [ c^{-2} X^{2}(t) ] \\
 \, &= \, \,  t^{2\alpha - \gamma+ 1} \cdot \bigg(  \int^{1}_{0} ( 1-v)^{2\alpha } v^{-\gamma} {\mathrm d} v + \int^{\infty}_{0} [  ( 1+v)^{\alpha} v^{-\gamma/2} - v^{\alpha-\gamma/2}] ^{2} {\mathrm d} v \bigg) \, =\, \kappa \,  t^{2H}  ,  
\end{split}
\end{equation}
where the last equality can be verified by Mathematica.  
\end{proof}

\begin{remark}[Integrability] \label{rem: 2.2}
{\em
In \eqref{eq: varcX} the indefinite integral over the infinite interval $(0, \infty) $ appears. Its integrability is verified under the condition \eqref{def-X-g-c} by direct calculation.  
 As discussed above, we consider the two cases  $\, 0 < \alpha <  (1+\gamma)/2\,$ and $\, - (1-\gamma)/ 2 < \alpha < 0 \,$ for a gamma $\gamma \in (0,1)$. 
 
Indeed, if 
$\, 0 < \alpha <  (1+\gamma)/2\,$, then $\,u \mapsto (1+u)^{\alpha} - u^{\alpha}  \,$ is a  decreasing function in $(0, \infty)$ 
and by Taylor expansion,  $\,(1+u)^{\alpha} - u^{\alpha}\, \le \, \alpha u^{\alpha-1}\,$ for every $\, u > 0 \,$, and hence, the indefinite integral in  \eqref{eq: varcX} is integrable: 
\begin{equation} \label{eq:integrability0}
\begin{split}
& \int^{\infty}_{0} [ (1 + u)^{\alpha} - u^{\alpha}]^{2} u^{-\gamma}{\mathrm d} u \, =\,  \Big(\int^{1}_{0} + \int^{\infty}_{1}\Big) [ (1 + u)^{\alpha} - u^{\alpha}]^{2} u^{-\gamma}{\mathrm d} u  \, \\
& \le \int^{1}_{0} u^{-\gamma} {\mathrm d} u + \int^{\infty}_{1} \alpha^{2} u^{2(\alpha - 1) - \gamma} {\mathrm d} u \, =\,  \frac{1}{\, 1 - \gamma \, } + \frac{\,\alpha^{2}\,}{\,1 + \gamma - 2 \alpha \,}\, =:\, c_{1} < \infty . \,  
\end{split}
\end{equation}
Similarly, if $\, - (1-\gamma)/ 2 < \alpha < 0 \,$, then set $\, \widetilde{\alpha} := - \alpha \in (0, 1/2) \,$, $ \widetilde{\gamma} :=  - 2 \alpha + \gamma \in (0, 1)  $. Rewriting the indefinite integral in \eqref{eq: varcX} in terms of $ \widetilde{\alpha}$ and $ \widetilde{\gamma}$, we obtain  
\begin{equation} \label{eq:integrability1} 
\begin{split}
& \int^{\infty}_{0} [ (1 + u)^{\alpha} - u^{\alpha}]^{2} u^{-\gamma}{\mathrm d} u \, =\,  \int^{\infty}_{0} \frac{\,[ (1 + u)^{-\alpha} - u^{-\alpha}]^{2} \,}{\,u^{-2\alpha} ( 1 + u)^{-2\alpha} \,}  u^{-\gamma}{\mathrm d} u \, \\
& \le \, \int^{\infty}_{0} {\,[ (1 + u)^{ \widetilde{\alpha}} - u^{ \widetilde{\alpha}}]^{2} \,}
u^{- \widetilde{\gamma}}{\mathrm d} u \le \frac{1}{\, 1 - \widetilde{\gamma}\, } + \frac{\, \widetilde{\alpha}^{2}\,}{\,1 + \widetilde{\gamma} - 2 \widetilde{\alpha}\,} < \infty \, ,  
\end{split}
\end{equation}
where we used \eqref{eq:integrability0} with $\alpha, \gamma $ being replaced by $ \widetilde{\alpha}, \widetilde{\gamma}$ in the second inequality. If $\, \alpha \, =\,  0 \,$, the indefinite integral in \eqref{eq: varcX} is $0$. Thus, under the condition \eqref{def-X-g-c}, the GFBM $X$ is well defined. } 
\end{remark}

For the standard FBM $B^H$, we have
$\Var(c^{-1}B^H(t)) =   t^{2H}$, which is increasing in $H$ for each $t>0$, and is also increasing in $t$ for each $H$. 
It is clear that  the same properties hold for the process $X$.  In addition, we observe that $\Var(X(t))$ is decreasing in $\gamma$ for each $t>0$.

\begin{remark}[Standard FBM]
{\em  By the recursion formula of the gamma function, when $\gamma = 0 $ and $\alpha \in (-1/2, 1/2)$, it is the standard FBM $B^H$ and the constant $\kappa $ in \eqref{eq: kappa1} is reduced to $\kappa = \Gamma ( 1 + \alpha ) \Gamma ( 1 - 2 \alpha) / [(1+2\alpha ) \cdot \Gamma ( 1 - \alpha ) ] $. } 
\end{remark}

\subsection{Generalized Riemann-Liouville (R-L) FBM}  \label{sec-G-RL-FBM}
A special model is the generalized Riemann-Liouville (R-L)  FBM, introduced in Remark 5.1 in \cite{pang-taqqu}. It is defined by 
\begin{equation}
\label{eqn-GFBM-RL-X}
X(t) = c \int_0^t (t-u)^\alpha u^{-\gamma/2} B( {\mathrm d} u), \quad  t\ge 0,
\end{equation}
where $B( {\mathrm d} u)$ is a Gaussian random measure on $\RR$ with the Lebesgue control measure $ du$ and
$$
c \in \RR, \quad  \gamma \in [0,1), \quad
\alpha \in \Big(-\frac{1}{2}+ \frac{\gamma}{2}, \ \frac{1}{2}+ \frac{\gamma}{2}\Big).
$$
Such a process  is also  a continuous self-similar Gaussian process with Hurst parameter $H = \alpha - \frac{\gamma}{2}+ \frac{1}{2} \in (0,1)$.  Equivalently, given the Hurst parameter $H \in (0,1)$ and a parameter $\gamma \in [0,1)$, the process $X$ in \eqref{eqn-GFBM-RL-X} can be represented as 
\begin{equation}
\label{eqn-GFBM-RL-X-alternative}
X(t) = c \int_0^t (t-u)^{H- \frac{1}{2} +\frac{\gamma}{2} } u^{-\frac{\gamma}{2}} B( {\mathrm d} u), \quad  t\ge 0. 
\end{equation}
When $\gamma=0$, this becomes the standard R-L FBM:
\begin{equation}
\label{eqn-GFBM-RL-X-alternative}
B^H(t) = c \int_0^t (t-u)^{H- \frac{1}{2}  }  B( {\mathrm d} u), \quad  t\ge 0. 
\end{equation}
This process was introduced by L\'evy (\cite{Levy}, see also Chapter 6 in \cite{PT17};  modulo some constant scaling). 
{
When $H=1/2$, i.e., $\alpha =\frac{\gamma}{2}$ for $\gamma \in [0,1)$, 
\begin{equation} \label{eqn-X-GRL-H=1/2}
X(t) = c\int_0^t (t/u -1)^{\gamma/2} B( {\mathrm d}u) = c \int^{t}_{0} ( t / u -1)^{\alpha} B( {\mathrm d} u ),  \quad t \ge 0. 
\end{equation}
If $\gamma=0$, then $X(t) =c B(t)$ is a Brownian motion, but if $\gamma \, =\, 2 \alpha \, \in (0,1)$, the increment can be rewritten as 
\[
X(t) - X(s) \, =\,  c \int^{s}_{0} ((t-u)^{\alpha} - (s-u)^{\alpha}) u^{-\alpha} B({\mathrm d} u) + \int^{t}_{s} ( t - u)^{\alpha} u^{-\alpha} B ({\mathrm d} u ) \, , \quad 0 \le s < t \, ,  
\]
 indicating the non-stationarity of its increments. 
}

The process $X$ in \eqref{eqn-GFBM-RL-X} has the covariance function
\begin{align}
\Psi(s,t) = E[X(s)X(t)] = c^2 \int_0^s (s-u)^{\alpha} (t-u)^{\alpha} u^{-\gamma} {\rm d}u, 
\end{align}
and the second moment of its increment
\begin{align} \label{Phi-GFBM-RL}
\Phi(s,t)  = E[(X(t) - X(s))^2] &= c^2 \int_s^t (t-u)^{2\alpha} u^{-\gamma} {\rm d}u + c^2 \int_0^s |(t-u)^{\alpha} - (s-u)^{\alpha} |^2 u^{-\gamma} {\rm d}u, 
\end{align}
for $0 \le s\le t$. 

We  also have the variance
\begin{align*}
\Var(c^{-1}X(t)) &= \int_0^t (t-u)^{2\alpha} u^{-\gamma} {\rm d}u \\
&= t^{2\alpha - \gamma +1} \int_0^1 (1-v)^{2\alpha} v^{-\gamma} {\rm d}v =  t^{2\alpha - \gamma +1}\text{Beta} (1-\gamma, 2\alpha+1). 
\end{align*}
Thus, with normalization constant $c=c(\alpha, \gamma) = \big(\text{Beta} (1-\gamma, 2\alpha+1)\big)^{-1/2}$, we have 
$\Var(X(t)) =  t^{2\alpha-\gamma+1} = t^{2H}$ for $t\ge 0$. 
It is clear that with this normalization factor, we have $\Var(X(t))$ increasing in $\alpha$ and decreasing in $\gamma$  for each fixed $t>0$.

\section{H{\"o}lder continuity} \label{sec-Holder}

In this section we prove the H{\"o}lder continuity property of the GFBM $X$ in \eqref{def-X}.  For convenience, we assume from now on that the GFBM X is
normalized with $c=c(\alpha, \gamma)$ given in \eqref{eqn-c-def}. 


Recall that for FBM $B^H$ with Hurst parameter $H \in (0,1)$, by self-similarity, we have 
$$
E\big[ |B^H(t) - B^H(s)|^p \big] = |t-s|^{pH} E\big[ |B_H(1)|^p \big]
$$
for any $p>0$.  Then the H{\"o}lder continuity property follows from applying the Kolmogorov-Centsov continuity criterion. Namely, the FBM $B^H$ admits a version whose sample paths are almost surely H{\"o}lder continuous of order strictly less than $H$.

Due to the lack of stationary increments, the proof of  the H{\"o}lder continuity property of the GFBM $X$ requires a delicate study of the second moment of the increment.

\begin{theorem}\label{thm-Holder}
For every $\, T > 0 \,$, there exists a positive constant $\,C_{T}\,$ such that the covariance function $\,\Phi \,$ in \eqref{eqn-Phi} satisfies $\, \Phi(s, t)  \le C_{T} \lvert s  - t \rvert^{2H}  \,$ for $ 0 < s < t \le T\,$, and hence,  by Kolmogorov-Centsov continuity criterion, the sample path $\, t \mapsto X(t) \,$ of the Gaussian process $\, X \,$ in \eqref{def-X} is almost surely $\,\alpha_{0}\,$-H\"older continuous for $\, 0 \le t \le T\,$ with $\, 0 < \alpha_{0} < H = (2 \alpha - \gamma+1)/2\,$. 
\end{theorem} 

\begin{proof} When $\, \gamma \, =\,  0 \,$, it is the case of FBM with Hurst index $\, \alpha + 1/2 \,   \in (1/2, 1)\,$, and the result of H\"older continuity is known. Thus, let us consider the case with $c = 1$ and $\gamma \neq 0 $.  
  First, let us decompose $\, X(t)- X(s) \,$ from \eqref{X-rep-1} into three independent components $\mathcal C_{1}, \mathcal C_{2}, \mathcal C_{3}$: 
\begin{equation}\label{eq: 4}
\begin{split} 
X(t)-X(s) & \, =\,  \int_{\mathbb R} \big( ( t-u)_{+}^{\alpha} - ( - u )_{+}^{\alpha} \big) \lvert u\rvert^{-\gamma/2 } B({\mathrm d} u ) -  \int_{\mathbb R}\big( ( s-u)_{+}^{\alpha} - ( - u )_{+}^{\alpha} \big) \lvert u\rvert^{-\gamma/2 } B({\mathrm d} u )\\
 &\, =:\,  \mathcal C_{1} +   \mathcal  C_{2} +  \mathcal C_{3} \, , 
 \end{split} 
 \end{equation}
 where 
 \begin{equation}
 \begin{split}
 \mathcal C_{1} &\, :=\,\int^{s}_{0} [( t-u)^{\alpha} - (s-u)^{\alpha} ] u^{-\gamma/2} B({\mathrm d}u) \, , \\
 \mathcal C_{2} &\, := \, \int^{t}_{s} ( t-u)^{\alpha} \lvert u \rvert^{-\gamma/2} B({\mathrm d} u)\,, \\ \quad  \mathcal  C_{3}  &\, := \,   \int_{-\infty}^0 \big( ( t-u)_{+}^{\alpha} - ( - u )_{+}^{\alpha} \big) \lvert u\rvert^{-\gamma/2 } B({\mathrm d} u ) \\
 & \,\,\,=\, \int^{\infty}_{0} [  ( t+u)^{\alpha} - (s+u)^{\alpha}]  u^{-\gamma/2} \widetilde{B}({\mathrm d} u) \,  
\end{split}
\end{equation}
with $\widetilde{B}$ being given in \eqref{eqn-wt-B}.  
Thus, we have $$\,\Phi (s, t) \, =\,  E \big[ (X(t)-X(s))^{2}\big] \, =\,   E \big[ \mathcal C_{1}^{2} + \mathcal C_{2}^{2} + \mathcal C_{3}^{2}\big] \,,$$ 
and hence, we shall evaluate squared expectation of these three terms separately.  
{It is worth noting that all the three summands are basically of the same order (all are less than or equal to $C|t-s|^{2H}$ for some $C>0$). }

\medskip

\noindent $\,\bullet \,$ {\bf (Evaluation of $\,  \mathcal  C_{1}\,$).} 
By the change of variables with $\, u \, =\,  s - (t-s)v \,$ and $\,v = x w \,$, $\,x = s/ (t-s)\,$ for every $s, t $ with $ 0 < s < t \le T$, we have 
\begin{equation} \label{eq:C1-1}
\begin{split}
E [ {\mathcal C}_{1}^{2} ] \, &=\, \int^{s}_{0} \lvert ( t - u)^{\alpha} - (s- u)^{\alpha} \rvert^{2} u^{-\gamma} {\mathrm d} u \\
& =  \int^{s/(t-s)}_{0} (t-s)^{2\alpha + 1 - \gamma} \Big( \frac{\,s\,}{\,t-s\,} - v\Big)^{-\gamma} ( ( 1 + v)^{\alpha} - v^{\alpha})^{2} {\mathrm d} v \\
&= (t-s)^{2H} \Big( \int^{1}_{0} ( 1 - w)^{-\gamma} (( 1 + x w)^{\alpha}- (xw)^{\alpha})^{2} x^{1-\gamma}  \,  {\mathrm d} w \Big) \Big \vert_{ x = s/(t-s) } \,. 
\end{split}
\end{equation}
{
When $\, 0 < \alpha < (1+\gamma)/2\,$, we have 
\begin{equation} \label{eq:C1-2}
\begin{split}
& \int^{1}_{0} ( 1 - w)^{-\gamma} (( 1 + x w)^{\alpha}- (xw)^{\alpha})^{2} x^{1-\gamma}  \,  {\mathrm d} w \\
 = & \int^{1}_{0} ( 1 - w)^{-\gamma} w^{\gamma-1} (( 1 + x w)^{\alpha}- (xw)^{\alpha})^{2} (xw)^{1-\gamma}  \,  {\mathrm d} w \\
\le & \, \sup_{y > 0 } \big\{ ( (1 +y)^{\alpha} - y^{\alpha})^{2} y^{1-\gamma}  \big\} \cdot \int^{1}_{0} ( 1 - w)^{-\gamma} w^{\gamma-1} {\mathrm d} w  \\
 \le & 4 \,  \text{Beta} ( 1 - \gamma , \gamma) < \infty \, ,  
\end{split}
\end{equation}
where we used the inequality 
\begin{equation} \label{eq: C1-2-1}
\, \sup_{y > 0 } ( (1 +y)^{\alpha} - y^{\alpha})^{2} y^{1-\gamma} \le \max \Big\{  \sup_{y \ge 1 } ( (1 +y)^{\alpha} - y^{\alpha})^{2} y^{1-\gamma} ,  \sup_{0 < y \le 1 } ( (1 +y)^{\alpha} - y^{\alpha})^{2} y^{1-\gamma} \Big\}  \le 4 \, .   
\end{equation} 
To verify \eqref{eq: C1-2-1}, firstly we use $\, (1 + y)^{\alpha} - y^{\alpha} \le \alpha \,  y^{\alpha - 1} \,$, $\, y > 0 \,$, $\, \alpha > 0 \,$ to obtain 
$$\,( (1 +y)^{\alpha} - y^{\alpha})^{2} y^{1-\gamma}  \le \alpha^{2} y^{2(H-1)} \le \alpha^{2} < 1\, \quad \text{for} \quad \, y \ge 1 \,,$$
 and secondly, we evaluate 
 $$\,( (1 +y)^{\alpha} - y^{\alpha})^{2} y^{1-\gamma}  \le (1 +y)^{2\alpha} y^{1-\gamma} \le 4^{\alpha} \le 4  \, \quad \text{for} \quad  \, 0 < y \le 1\,,$$ and then combine the inequalities. 

Similarly, when $\, -(1-\gamma)/2 < \alpha < 0   \,$, considering $ \widetilde{\alpha} = - \alpha $, $ \widetilde{\gamma} = - 2 \alpha + \gamma < 1 $, as we derived in (\ref{eq:integrability1}), and also using a similar inequality to \eqref{eq: C1-2-1} (but now with $\,\widetilde{\alpha}\,$ and $\,\widetilde{\gamma}\,$, instead of $\, \alpha\,$, $\, \gamma\,$), we obtain the upper bound for every $\,x = s/ (t-s) > 0 \,$, 
\begin{align}\label{eq:C1-3}
& \int^{1}_{0} ( 1 - w)^{-\gamma} (( 1 + xw)^{\alpha} - (xw)^{\alpha})^{2} x^{1 - \gamma} {\mathrm d} w \nonumber\\
& =\,  \int^{1}_{0} ( 1 - w)^{-\gamma} \frac{(( 1 + xw)^{ \widetilde{\alpha}} - (xw)^{ \widetilde{\alpha}})^{2}}{( 1 + xw)^{2 \widetilde{\alpha}} (xw)^{2 \widetilde{\alpha}} } x^{1 - \gamma} {\mathrm d} w \nonumber\\
& \le \int^{1}_{0} ( 1 - w)^{-\gamma}w^{-2 \widetilde{\alpha} } {(( 1 + xw)^{ \widetilde{\alpha}} - (xw)^{ \widetilde{\alpha}})^{2}} x^{1 - \widetilde{ \gamma} } {\mathrm d} w \nonumber\\
& = \int^{1}_{0} ( 1 - w)^{-\gamma}w^{\gamma -1} {(( 1 + xw)^{ \widetilde{\alpha}} - (xw)^{ \widetilde{\alpha}})^{2}} (xw)^{1 - \widetilde{ \gamma} } {\mathrm d} w \nonumber\\
& \le\sup_{y> 0 } \big\{ ((1+y)^{ \widetilde{\alpha}} - y^{ \widetilde{\alpha}})^{2} y^{1- \widetilde{\gamma}} \big\} \cdot \int^{1}_{0} ( 1 - w)^{- \gamma}w^{\gamma-1 }  {\mathrm d} w \, \nonumber\\
& \le \, 4 \, \text{Beta} ( 1-\gamma, \gamma ) < \infty  \, . 
\end{align}
} 
Thus, combining (\ref{eq:C1-2})--(\ref{eq:C1-3}) with (\ref{eq:C1-1}), we claim that there exists a positive constant $\,c_{3}\,$ such that 
\begin{equation}
\label{eq:C1-4}  
 E [ {\mathcal C}_{1}^{2} ] \le c_{3} (t-s)^{2H}\,. 
\end{equation}

\medskip 

\noindent $\, \bullet \,$ {\bf (Evaluation of $\, \mathcal  C_{2}\,$).} Similarly, for the second term in (\ref{eq: 4}), by the change of variables with $\,u \, =\,  (t-s)v + s\,$ and $\, H\, =\, 2\alpha - \gamma + 1 \,$, we obtain for $\, 0 \le s < t < \infty\,$, 
\begin{equation} \label{eq:C2}
\begin{split}
 E [  \mathcal  C_{2}^{2} ] & \,=\,\int^{t}_{s} ( t- u)^{2\alpha } u ^{-\gamma} {\mathrm d} u \, \le\,  ( t- s)^{2\alpha - \gamma +1} \int^{1}_{0} (1-v)^{2\alpha} v^{-\gamma} {\mathrm d} v \, =\,  c_{4} \lvert t - s\rvert^{2H} \, , 
\end{split}
\end{equation}
where $\, c_{4} \, :=\, \text{Beta} ( 1 + 2\alpha, 1 - \gamma) \,$. This holds for every $\, \alpha > -1/2\,$.

\medskip 

\noindent $\, \bullet \,$ {\bf (Evaluation of $\, \mathcal  C_{3}\,$).} For the third term $\,  {\mathcal C}_{3}\,$ in (\ref{eq: 4}), when $\,0 < \alpha < (1+\gamma)/2\,$, because of  (\ref{eq:integrability0}) in Remark \ref{rem: 2.2}, we have with the constant $\, c_{1}\,$ in (\ref{eq:integrability0}), for $\, s < t \,$, with $\,x \, :=\, s / (t-s) > 0 \,$, 
\begin{align} \label{ineq: bdd int1}
& \int^{\infty}_{x}  ( (1+v)^{\alpha} - v^{\alpha} ) ^{2} ( v - x )^{-\gamma} {\mathrm d} v  \nonumber \\
& \le \int^{\infty}_{x}  (( 1+v - x ) ^{\alpha} - ( v - x) ^{\alpha}) ^{2} ( v - x)^{-\gamma} {\mathrm d} v  \nonumber \\
& 
\le \int^{\infty}_{0} ( (1 + u)^{\alpha} - u^{\alpha})^{2} u^{-\gamma}{\mathrm d} u \, \le \,  c_{1} . 
\end{align}
{
Similarly, when $\, - (1- \gamma)/2 < \alpha < 0 \,$, again with $\, \widetilde{\alpha} \, =\, - \alpha\,$ and $\, \widetilde{\gamma} \, =\, - 2 \alpha + \gamma < 1\,$, for every $s, t $ with $0 < s < t $ and $\,x := s / (t-s) > 0 \,$, we have 
\begin{align} \label{ineq: bdd int2}
& \int^{\infty}_{x}  ((1+v)^{\alpha} - v^{\alpha} ) ^{2} ( v - x )^{-\gamma} {\mathrm d} v \nonumber \\
& = \int^{\infty}_{x}  \frac{ (1+v)^{ \widetilde{\alpha}} - v^{ \widetilde{\alpha}} ) ^{2}}{( 1+v)^{2 \widetilde{\alpha}} v^{ -2 {\alpha}}} \cdot  ( v - x )^{-\gamma} {\mathrm d} v \non \\
& \le \int^{\infty}_{x}  {( (1+v)^{ \widetilde{\alpha}} - v^{ \widetilde{\alpha}} ) ^{2}}  \cdot ( v - x)^{2\alpha -\gamma} {\mathrm d} v \nonumber  \\
& \le \int^{\infty}_{0} ( ( 1 + u)^{ \widetilde{\alpha}} - u ^{ \widetilde{\alpha}})^{2} u^{- \widetilde{\gamma}} {\mathrm d} u < \infty \, ,  
\end{align}
where we used \eqref{eq:integrability1} in the last part of inequalities.  
}

Then for $ 0 < s < t $, by changing the variables 
with $\, u = (t-s)v - s \,$ 
and then using \eqref{ineq: bdd int1}-(\ref{ineq: bdd int2}) separately, we claim that there exists a positive constant $\,c_{5}\,$ such that 
\begin{equation} \label{eq: inf integral0-H}
\begin{split}
E \big[ \mathcal C_{3}^{2}\big] \, &=\, \int^{\infty}_{0} [ (t + u)^{\alpha} -  (s+u)^{\alpha}] ^{2} u^{-\gamma}  {\mathrm d} u \\
& =\,  
(t-s)^{2H} \int^{\infty}_{s/(t-s)}  [ (1+v)^{\alpha} - v^{\alpha} ] ^{2} \Big( v - \frac{\,s\,}{\,t-s\,} \Big)^{-\gamma} {\mathrm d} v\\
& \le c_{5} (t-s)^{2H}.
\end{split}
\end{equation}

Combining these inequalities (\ref{eq:C1-4}), (\ref{eq:C2}), (\ref{eq: inf integral0-H}) with (\ref{eq: 4}), we obtain the desired inequality,  because the second moments $E [ \lvert X(t)\rvert^{2}] $ are finite as it is given in (\ref{eq:VarXt}) and {for $\, 0 < s <  t < T$,    }
\begin{equation}\, 
\Phi(t, s) \, \le\, (c_{3} + c_{4} + c_{5}) \cdot  \lvert t - s \rvert^{2\alpha - \gamma + 1} \, . 
\end{equation} 

Since $\, X \,$ is a zero-mean Gaussian process, $\, X(t) - X(s) \,$ is a Gaussian random variable with mean $\, 0 \,$ and variance $\, \Phi (s, t)\,$, and hence, its $\,2p\,$-th moment ($\, p \ge 1 \, $) can be evaluated by 
\[
E [ \lvert X (t) - X(s) \rvert^{2p}] \le c_{p} [\Phi (t, s) ]^{p} \le c_{p} C^{p} \lvert t - s \rvert^{(2\alpha - \gamma +1)p} \, 
\]
for some positive constant $\,c_{p}\,$ which depends on $\,p\,$. Then applying the Kolmogorov-Centsov continuity criterion (e.g., Theorem 1.2.1 of \textsc{Revuz \& Yor} (1991)), we conclude that the sample paths of  the GFBM $\, X\,$ in \eqref{def-X} is $\,\alpha_{0}\,$-H\"older continuous on every finite interval $\,[0, T]\,$ with probability one for $\, 0 < \alpha_{0} < H= (2 \alpha - \gamma+1)/2\,$. 
\end{proof}

\begin{remark}
{\em
When $\gamma$ is close to $1$ and $\alpha > 0$, the Hurst parameter $H$ can be chosen with $H < 1/2$. Thus, Theorem \ref{thm-Holder} covers the whole range of $\, H \in (0, 1)\,$.  

For the generalized R-L FBM $X$ in \eqref{eqn-GFBM-RL-X}, the same H{\"o}lder continuity property holds. }
\end{remark}

\begin{remark} 
{\em 
Consider   $\, X(t) \, =\,  \int^{t}_{0} \kappa(t, u ) B({\mathrm d} u) \,$, $\, t \ge 0 \,$ with a Volterra kernel $\,\kappa (t,u) \, =\,  (t-u)^{\alpha} u^{-\gamma/2}  \,$.
 This process or similar processes have been recently studied by Yazigi (2015) \cite{yazigi2015representation}.
By Theorem 2.1 in \cite{yazigi2015representation}, the Volterra kernel $\kappa$ can be written as
\[
k(t, u) \, =\,  t^{H - 1/2} F(u/t) \, ; \quad t \ge  0 \, , 0 \le u \le t , 
\]
where $H=\alpha - \gamma/2+1/2$ is the Hurst parameter, and $F(v) = (1-v)^\alpha v^{-\gamma/2}$ and $F(v)\equiv 0$ for $v>1$ (it is clear that $F \in L^2(\RR_+, du)$, i.e., $\int_{\RR} |F(u)|^2 du < \infty$.) 
In the related work \cite{azmoodeh2014necessary}, necessary and sufficient conditions are derived for H\"older continuity of such self-similar processes. 
The condition involves the function $\Phi(s,t)$, and is closely related to the Fernique's theorem on the continuity of Gaussian processes. 
By Theorem 1 in \cite{azmoodeh2014necessary}, we obtain that there exist constant $c_\epsilon$ such that the function $\Phi(s,t)$ in \eqref{Phi-GFBM-RL} satisfies
$$
\Phi(s,t)^{1/2}  \le c_{\epsilon} |t-s|^{H-\epsilon}, \quad \text{for all} \quad \epsilon>0. 
$$

On the other hand, the proof of  Theorem 1 in \cite{azmoodeh2014necessary}
 uses the Garsia-Rademich-Rumsey inequality (see Lemma 2 in \cite{azmoodeh2014necessary}), which unfortunately only holds on the finite time interval $[0,T]$.  We are unable to prove the H{\"o}lder continuity property with that approach for the GFBM $X(t)$ in \eqref{def-X}. 
}
\end{remark}

\begin{remark}
{\em 
For standard FBM $B^H$, it is shown in 
Theorem 1.6.1 in \cite{biagini2008stochastic}
that the local  Law of Iterated Logarithm holds:
\begin{equation} \label{eqn-FBM-LLIL1}
\limsup_{t\to 0^+} \frac{|B^H(t)|}{t^H\sqrt{\log\log t^{-1}}} = c_H
\end{equation}
with probability one for some appropriate constant $c_H>0$.
 This implies that $B^H$ cannot have sample paths with H{\"o}lder continuity of order greater than $H = \alpha + 1/2$. 
 For the GFBM $X$ in \eqref{def-X}, we establish the local Law of Iterated Logarithm in Section \ref{sec-LLIL}, see Theorem \ref{thm-LLIL}. That result will imply that 
 the process $X$ cannot have sample paths with H{\"o}lder continuity of order greater than $H= \alpha - \gamma/2 + 1/2$. } 
 
\end{remark}

\section{Path Differentiability} \label{sec-nondiff}

We prove the following differentiability/non-differentiability property of the sample paths of $X$. For FBM $B^H$ with $\,H \in (0, 1)\,$, the path non-differentiability property was established in \cite{mandelbrot1968fractional}; see also \cite[Proposition~1.7.1]{biagini2008stochastic}.   
The proof of  \eqref{eqn-diff-1} for FBM $B^H$ uses its self-similarity and stationary-increment properties, in particular, for $t>s$, the law of the ratio  ${(B^H(t) - B^H(s))}\, /\, {(t-s)}$ is the same as $(t-s)^{H-1} B^H(1)$, and the probability of $\{|B^H(1)|>a t_n^{1-H}\}$ converges to zero where $a>0$ is any positive constant and $\{t_n\}$ is any sequence decreasing to zero as $n\to\infty$. 
Distinct from the standard FBM $B^H$, the GFBM $X$ has a region of parameters $\,1/2<\alpha < 1/2+\gamma/2\,$ and $\gamma \in (0,1)$, which gives $H\in (1/2,1)$, in which its paths are differentiable, while in the rest of the parameter regions of $(\alpha,\gamma)$, resulting $H \in (0,1)$, its paths are non-differentiable. Recalling Remark \ref{rem-H=0.5}, in the case of $H=1/2$, we remark that the paths of $X$ are non-differentiable regardless of whether  $X$ is a Brownian motion ($\gamma=0$) or not ($\gamma \in (0,1)$).

\begin{theorem} \label{thm-nondiff}
The following differentiability properties hold:
\begin{itemize}
\item[(i)]
If $\,-1/2+\gamma/2<\alpha \le 1/2\,$ and $\gamma \in (0,1)$ (resulting in $H\in (0,1)$), the GFBM $X$ in \eqref{def-X} is not mean square differentiable and does not have differentiable sample paths; In particular, for every $s\in \RR_+$, 
\begin{equation} \label{eqn-diff-1}
\limsup_{t\to s} \left| \frac{X(t) - X(s)}{t-s}\right| = + \infty
\end{equation}
with probability one. 
\item[(ii)]
If $\,1/2<\alpha < 1/2+\gamma/2\,$ and $\gamma \in (0,1)$ (resulting in $H\in (1/2,1)$), the sample path of $X$ in \eqref{def-X} is continuously differentiable once with derivative 
\begin{equation} \label{eqn-diff-2}
\frac{\,{\mathrm d} X (t) \,}{\,{\mathrm d} t \,} \, =\, c  \int^{t}_{-\infty} \alpha ( t - s)^{\alpha - 1} \lvert s \rvert^{-\gamma / 2} {\mathrm d} B(s) \,, \quad t \ge 0 \,, 
\end{equation}
but not twice with probability one. 
\end{itemize}
\end{theorem}

\begin{proof}
We firs prove part (i). 
{ We apply Theorem of Yeh \cite{yeh1967differentiability} (see also its correction \cite{yeh1968correction}). It says that if a separable Gaussian process $\xi=\{\xi(t): t\in [0,T]\}$ has mean zero, and satisfies the Kolmogorov's continuity condition, and the lower bound: 
for some $\alpha, a>0$,
$$
E[|\xi(t)- \xi(s)|^2] \ge a |t-s|^\alpha, \quad t,s \in [0,T],
$$
then for $\lambda>\alpha/2$ and for any $t$,  
$$
\lim_{s\to 0} \sup\frac{|\xi( t\pm s) -\xi(t)|}{s^\lambda} =+\infty.
$$
}
For the differentiability of sample paths of (\ref{def-X}) we calculate a lower bound of $\,\Phi (s, t)
$
in (\ref{eqn-Phi}).

When $\, \gamma \in (0, 1)  \,$, $\, 0 < \alpha < 1/2  \,$, it follows from the calculation of H\"older continuity  for $\, 0 < s < t \,$, 
\begin{align*}
\Phi (s, t) 
& \ge c^{2}\int^{t}_{s} (t-u)^{2\alpha} u^{-\gamma} {\mathrm d} u \\
& \ge  c^{2 }\int^{(t+s)/2}_{s} (t-u)^{2\alpha} u^{-\gamma} {\mathrm d} u\\
& \ge \frac{c^{2}}{2 \alpha + 1 } \Big( \frac{\,t+s\,}{\,2\,} \Big)^{-\gamma} \cdot \Big( \frac{\,t-s\,}{\,2\,}\Big)^{2\alpha + 1} \,.  
\end{align*}
Thus, if $\, \alpha < 1/2\,$ with $\, 2 \alpha + 1  < 2\,$, by the Theorem of Yeh \cite{yeh1967differentiability},
the sample paths are almost nowhere differentiable in $\, [0, \infty) \,$.

Similarly, when  $\, \alpha \, =\,  1/2 \,$ and $\, \gamma \in (0, 1) \,$, we may compute directly for $\,0 < s < t < T \,$
\begin{equation*} 
\begin{split}
\Phi (s, t) & \ge c^{2} \int^{t}_{s} (t  - u ) u^{-\gamma} {\mathrm d} u \\
&\,=\, c^2  \bigg( \frac{\,t (t^{1-\gamma} -  s^{1-\gamma}) \,}{\, 1 - \gamma \,} - \frac{\,t^{2-\gamma} -  s^{2-\gamma}\,}{\,2-\gamma\,} \bigg) \\
& \ge \frac{\,c^{2}\,}{\,(1-\gamma) (2-\gamma) \,} \cdot (t - s)^{2 - \gamma} \, , 
\end{split}
\end{equation*}
since for every $\, t < T\,$, the function 
\[
\mathfrak h(s) \, :=\, \frac{\,t (t^{1-\gamma} -  s^{1-\gamma}) \,}{\, 1 - \gamma \,} - \frac{\,t^{2-\gamma} -  s^{2-\gamma}\,}{\,2-\gamma\,} - \frac{\,(t - s)^{2 - \gamma}
\,}{\,(1-\gamma) (2-\gamma) \,} \,, \quad 0 \le s \le t \, 
\]
of $\, s\,$ is nonnegative. Indeed, it has the first derivative 
\[
\mathfrak h^{\prime}(s) \, =\,  \frac{\,t-s\,}{\,1 - \gamma\,} \Big( (t-s)^{-\gamma} - (1 - \gamma) s^{-\gamma}\Big)
\]
and $\,\mathfrak h\,$ has a maximum at $\, s_{0} \, :=\, (1 + ( 1- \gamma)^{-1/\gamma})^{-1}\,  t \,$ with minima at $\,\mathfrak h(0) \, =\, 0 \, =\, \mathfrak h(t)  \,$. 
Thus, if $\, \alpha = 1/2\,$, $\,\gamma \in (0, 1) \,$ with $\, 2 -\gamma < 2\,$, again by the Theorem of Yeh \cite{yeh1967differentiability},
the sample paths are almost nowhere differentiable in $\, [0, \infty) \,$. Therefore, we conclude that the GFBM $\,X\,$ is almost nowhere differentiable in the parameter set $\,\alpha \le 1/2 \,$, $\, \gamma \in (0, 1) \,$.

Next we prove part (ii). 
If $\, \alpha > 1/2\,$, then the process $\, X\,$ 
is a semimartingale of finite variation with the following representation
\[
X(t) \, =\,  c \int^{t}_{0} \Big( \int^{t}_{-\infty} \alpha ( r - s)^{\alpha-1}_{+} \lvert s \rvert^{-\gamma/2} {\mathrm d} B(s) \Big) {\mathrm d} r, \quad t \ge 0 .  
\]
This follows from a stochastic version of Fubini theorem  (Theorem 4.6 of \cite{basse2009spectral}), because for $0 < r < t$, the stochastic integral $$h_{r,t}\, :=\, \int^{t}_{-\infty}\alpha ( r - s)^{\alpha-1}_{+} \lvert s \rvert^{-\gamma/2}{\mathrm d} B(s) 
\, =\, \int^{r}_{-\infty}\alpha ( r - s)^{\alpha-1}  \lvert s \rvert^{-\gamma/2}{\mathrm d} B(s) 
$$ 
is well defined. Indeed, we have   
\begin{align*}
\int^{t}_{-\infty} (r-s)_{+}^{2(\alpha - 1) } \lvert s \rvert^{-\gamma} {\mathrm d}s
& \, =\, \int^{r}_{-\infty} (r-s)_{+}^{2(\alpha - 1) } \lvert s \rvert^{-\gamma} {\mathrm d}s   \\
& \,=\,  \int^{r}_{0}  ( r - s)^{2(\alpha - 1) } s^{-\gamma} {\mathrm d} s + \int^{0}_{-\infty} ( r - s)^{2(\alpha - 1) } \lvert s \rvert^{-\gamma} {\mathrm d} s \\
& \, =\,  r^{2H}  \Big( \int^{1}_{0} (1-u)^{2(\alpha - 1) } u^{-\gamma} {\mathrm d} u + \int^{\infty}_{0} ( 1 + u)^{2(\alpha -1) } u^{-\gamma}  {\mathrm d} u \Big) < \infty \, . 
\end{align*}

 Since $\, (r, t) \mapsto h_{r, t}\, $ does not depend on $\,t\,$, we claim the sample path of $\, X \,$ is differentiable with its derivative $\,{\mathrm d} X(t) / {\mathrm d} t \, =\, c  h_{t, t}\,$, $\, t \ge 0 \,$ almost surely. 
However, $\, t \mapsto h_{t,t}\,$ is not differentiable with probability one, because of a similar reasoning. Indeed, for $s<t$, 
\[
 \int^{t}_{s}(t-u)^{2(\alpha-1)}_{+} u^{-\gamma} {\mathrm d} u  \,  \ge  \, \Big( \frac{\,t+s\,}{\,2\,} \Big)^{-\gamma} \cdot \Big( \frac{\,t-s\,}{\,2\,}\Big)^{2(\alpha-1) + 1} \,,  
 \]
and hence, by applying the result from \cite{yeh1967differentiability} again, we see $\, 2 (\alpha - 1)+1 < 2\,$ or equivalently, $\, \alpha < 3/2 \,$, the sample path of $\,h\,$ is not differentiable with probability one. 

Consequently, the sample paths of $\,X\,$ are continuously differentiable once but not twice almost surely for the fixed parameter $\, \alpha \in (1/2, (1+\gamma)/2)\,$. 
\end{proof}

\section{Functional Law of Iterated Logarithm} \label{sec-FLIL}

In this section we establish the functional  Law of Iterated Logarithm (FLIL) of the GFBM $X$ in \eqref{def-X}. We refer to \cite{{taqqu1977law}} and  \cite{taqqu1985survey} for the FLIL of FBM. 
We apply Theorem A1 in \cite{{taqqu1977law}} to prove the FLIL for the process $X$. We first introduce some notation and terminology. 

Let $\CC[0,T]$ be the space of continuous functions. Recall the covariance function $\Psi$ in \eqref{eqn-Psi}, which is shown to be continuous in \cite{pang-taqqu}. 
Let $\mathcal{H}(\Psi)$ be the reproducing kernel Hilbert space (RKHS) with $\Psi$ as the reproducing kernel. 
It is defined as the completion of the vector space spanned by the functions $\Psi(s, \cdot)$, $s \in [0,T]$, and endowed with the scalar product 
$$
\Big\langle \sum_i c_i \Psi(s_i, \cdot), \sum_j c'_j \Psi(s'_j,\cdot) \Big\rangle = \sum_i \sum_j c_i c'_j \Psi(s_i, s'_j). 
$$
Let $K:= \big\{h\in \mathcal{H}(\Psi): \langle h, h \rangle^{1/2} \le 1\big\}$ be the unit ball of $\mathcal{H}(\Psi)$. 
The FLIL in general states that
\begin{itemize}
\item[(a)] a certain sequence of functions $z_n$ of $\CC[0,T]$ is contained in an $\epsilon$-neighborhood of $K$ when $n$ is large (another way to say this, is that this sequence is relatively compact as $n\to\infty$, namely that $\{z_n\}$ contains a converging subsequence to an element of $K$), and
\item[(b)] the functions that are limiting points of the sequence $\{z_n\}$ fill up the set $K$. 
\end{itemize}
Let $d(\cdot,\cdot)$ be the sup-norm distance in $\CC[0,T]$, and $\CC\{z_n\}$ represents the cluster set {(the set of the limit points)} of the sequence $\{z_n\}$. 

{
The same properties in  \eqref{eqn-FLIL-1} and  \eqref{eqn-FLIL-2} below hold for the FBM $B^H$. They are stated in Corollary A1 in \cite{{taqqu1977law}}, applying Theorem A1 with the reproducing kernel $\Gamma_H(s,t) =  E\bigl[B^H(s) B^H(t)\bigr]$ in \eqref{eq: FBM cov.f}, and $K$ equal to the unit ball of $\mathcal{H}(\Gamma_H)$. 
}

\begin{theorem} \label{thm-FLIL}
Let $K$ be the unit ball of  the RKHS $\mathcal{H}(\Psi)$ with the covariance kernel $\Psi$ in \eqref{eqn-Psi}. 
The GFBM $X$ in \eqref{def-X} satisfies 
\begin{equation} \label{eqn-FLIL-1}
\lim_{n\to\infty} d \left( \frac{X(nt)}{ (2n^{2H} \log \log n)^{1/2}}, K \right) = 0, \quad a.s. 
\end{equation}
and
 \begin{equation} \label{eqn-FLIL-2}
\CC\left\{  \frac{X(nt)}{ (2n^{2H} \log \log n)^{1/2}}\right\} = K, \quad a.s. 
\end{equation}
where $H= \alpha - \gamma/2+1/2$. 
\end{theorem}

\begin{proof}
It is clear that $\Psi(t,t)$ is strictly increasing in $t$. 
We check the three conditions (C-1), (C-2) and (C-3) in Theorem A1 in  \cite{{taqqu1977law}}, that is,
\begin{itemize}
\item[(C-1)] 
\begin{align} \label{eqn-C1}
 \lim_{r\to\infty}\sup_{0\le s, t \le T} \left| \frac{E[X(rs)X(rt)] }{r^{2H}L(r)} - \Psi(s,t) \right| = 0. 
\end{align}
\item[(C-2)]
There is a nonnegative, strictly increasing and continuous function $\phi$ on $\RR_+$ satisfying $\int_1^\infty \phi(e^{-u^2})du <\infty$ such that 
\begin{align}\label{eqn-C2}
E\big[(X(rs) - X(rt))^2\big] \le \phi^2(|s-t|) r^{2H} L(r), \quad s, t \ge 0,\, r \ge 0.
\end{align}
\item[(C-3)]
\begin{align}\label{eqn-C3}
\lim_{n\to\infty, m/n\to\infty} E \left[ \frac{X(ms)}{m^HL^{1/2}(m)}\frac{X(ns)}{n^HL^{1/2}(n)} \right] = 0
\end{align}
\end{itemize}
We take $L(\cdot) \equiv 1$. 

For (C-1), for $r>0$ and $s<t$, we have
\begin{equation} \label{Cov-scale-identity}
E[X(rs)X(rt)] = \Psi(rs, rt)= r ^{2\alpha - \gamma +1}  \Psi(s,t),
\end{equation}
which immediately implies that 
$$
 \sup_{0\le s, t \le T} \left| \frac{E[X(rs)X(rt)] }{r^{2H}} - \Psi(s,t) \right| = 0. 
$$
Thus, \eqref{eqn-C1} holds. 

For (C-2),  as shown in Theorem \ref{thm-Holder} for  the H{\"o}lder continuity property, for $r>0$ and $s<t$, we have
$$
E\big[(X(rs)-X(rt))^2\big]  \le C_T |s-t|^{2H} r^{2H},
$$
for some constant $C_T>0$. 
Let $\phi(t) = C_T t^{2H}$ for $t \in [0,T]$. It is clear that $\phi(\cdot)$ is a nonnegative, strictly increasing and continuous function satisfying $\int_1^\infty \phi(e^{-u^2}) {\mathrm d} u = C_T \int_1^\infty  e^{- 2H u^2} <\infty$ since $H \in (0,1)$.  This verifies the condition (C-2). 

For (C-3), for $s<t$, and for $m>n>0$ satisfying $m/n\to\infty$ (noting that $mt>ns$),  we have
\begin{align} \label{eqn-Psi-nmst}
\Psi(ns, mt) &= c^2  \int_0^{ns} (mt-v)^{\alpha} (ns-v)^{\alpha} v^{-\gamma} \mathrm{d}v \nonumber \\
& \qquad + c^2  \int_0^{\infty} ((mt+u)^{\alpha} - u^{\alpha}) ((ns+u)^{\alpha} - u^{\alpha})u^{-\gamma}\mathrm{d}u. 
\end{align}
The first integral term is equal to
\begin{align}
\int_0^s (mt -nu)^\alpha (ns-nu)^\alpha (nu)^{-\gamma} n \mathrm{d}u
& = \int_0^s m^{\alpha}  \left(t- \frac{n}{m} u \right)^{\alpha} n^{\alpha +1 -\gamma} (s-u)^\alpha u^{-\gamma} \mathrm{d}u. \non
\end{align}
Dividing by $m^{H}n^{H}$, we obtain
\begin{align}
\left(\frac{n}{m}\right)^{(1-\gamma)/2} \int_0^s   \left(t- \frac{n}{m} u \right)^{\alpha}  (s-u)^\alpha u^{-\gamma} \mathrm{d}u \to 0 \non
\end{align}
as $n, m \to \infty$ and $n/m\to 0$. 

The second integral term in \eqref{eqn-Psi-nmst},  we have
\begin{align}
& m^{\alpha} n^{\alpha} \int_0^{\infty} ((t+u/m)^{\alpha} - (u/m)^{\alpha}) ((s+u/n)^{\alpha} - (u/n)^{\alpha})u^{-\gamma}\mathrm{d} u \non\\
&= m^{\alpha} n^{\alpha} \int_0^{\infty} \left[ \left( t+ \frac{n}{m} r\right)^\alpha 
- \left( \frac{n}{m}r\right)^\alpha \right] \left[ (s+r)^{\alpha} - r^{\alpha} \right] n^{-\gamma}r^{-\gamma}n \mathrm{d} r.\non
\end{align}
Dividing by $m^{H}n^{H}$, we obtain
\begin{align}
 \left(\frac{n}{m}\right)^{(1-\gamma)/2} \int_0^{\infty} \left[ \left( t+ \frac{n}{m} r\right)^\alpha 
- \left( \frac{n}{m}r\right)^\alpha \right] \left[ (s+r)^{\alpha} - r^{\alpha} \right] r^{-\gamma}\mathrm{d} r \to 0 \non
\end{align}
as $n, m \to \infty$ and $n/m\to 0$. 

Thus, for $s<t$, 
\begin{align}
\lim_{n\to\infty, m/n\to \infty}E\left[ \frac{X(ns)}{n^H}  \frac{X(mt)}{m^H} \right] & =  \lim_{n\to\infty, m/n\to \infty} \frac{1}{n^H m^H} \Psi(ns, mt) =0. \non
\end{align}
For the case $s>t$, we can switch $s$ and $t$ above in \eqref{eqn-Psi-nmst}, and note that we can let $m>>n$ such that $mt>ns$. Then the same argument will follow.   Therefore we have verified \eqref{eqn-C3} in Condition (C-3).  This completes the proof. 
\end{proof}

\section{Local law of the iterated logarithms} \label{sec-LLIL}

For FBM $B^H$, the local  Law of Iterated Logarithm states that with probability one,
$$
\limsup_{u\to 0^+} \frac{|B^H(ut)|}{u^H\sqrt{\log\log u^{-1}}} = c_H
$$
for all $t \in (0,T]$, as in \cite{arcones1995law}. See the equivalent expression in \eqref{eqn-FBM-LLIL1}. 
For Gaussian processes, Arcones \cite{arcones1995law} has established a useful criterion to prove the local law of the iterated logarithm in Theorem 4.1.  We apply that to prove the following for the GFBM $X$ in \eqref{def-X}. 

\begin{theorem} \label{thm-LLIL}
For the GFBM $X$ in \eqref{def-X}, 
if $\,\alpha > 0 \,$, with probability one, 
$$
\limsup_{u\to 0^+} \frac{|X(u t)|}{u^H\sqrt{\log\log u^{-1}}} 
$$
exists for all $t \in (0,T]$.

\end{theorem}

\begin{proof}
We check the nine conditions in  \cite[Theorem 4.1]{arcones1995law}. 
Here we consider the interval $[0,T]$ and use natural pseudometric $\rho(s,t) = \sqrt{E[(X(s) - X(t))^2]} = \sqrt{\Phi(s,t)}$. 
 Let $\tau(u) =u$ and $w(u) = u^{H}$. 
Condition (i) is evident and condition (ix) is clear since $\Phi(s,t)$ is continuous. For (v), $([0,T], \rho)$ is totally bounded, since $\rho(0,T) < \infty$. 
It is clear that the conditions in (vii) and (viii) holds since these functions are continuous. Condition (ix) holds since $\Phi(s,t)$ is continuous. 


For (ii), by the scaling identity of the covariance in \eqref{Cov-scale-identity}, we obtain
$$
E \left[ \frac{X(\tau(u) s) X(\tau(u) t)}{w(u)^2} \right]  = E \left[ \frac{X(u s) X(u t)}{u^{2H}} \right] = \Psi(s,t). 
$$

For (iii), we shall show that for each $m\ge 1$, each $\ep>0$, each $t_1,\dots,t_m\in (0,T]$ and each $\lambda_1, \dots, \lambda_m \in \RR$, 
\begin{align} \label{eqn-LLIL-p-iii}
\lim_{r\to 1-} \limsup_{u \to 0+} \sup_{v: u e^{-(\log u^{-1})^r} \le v \le u e^{-(\log u^{-1})^\ep}} \sum_{j,k=1}^m \lambda_j \lambda_k E \left[ \frac{X(ut_j) X(v t_k)}{w(u)w(v)} \right] \le 0.
\end{align}

For $r > 0 $, let $\varphi(u) :=  u e^{-(\log u^{-1})^r}<u$.  Note that $\varphi(u)/u=e^{-(\log u^{-1})^r} \to 0$ as $u \to 0+$ and for $u > 0 $, $\varphi(u)/u=e^{-(\log u^{-1})^r} \to u  $ as $r \to 1-$. 
Consider for $t>s$, 
\begin{align}
& E \left[ \frac{X(ut) X(\varphi(u)s)}{u^H \varphi(u)^H} \right]  = \frac{1}{u^H \varphi(u)^H} \Psi(ut, \varphi(u) s), \non
\end{align}
where $\Psi(\cdot, \cdot) $ has two terms as in (\ref{eqn-Psi}). 
By the change of variables from $v$ to $\theta  \varphi (u) s$, we have
\begin{align} \label{eq: Thm6.1-2}
&  \frac{1}{u^H \varphi(u)^H}\int_0^{\varphi(u)s} (ut-v)^{\alpha} (\varphi(u)s-v)^{\alpha} v^{-\gamma} {\mathrm d} v \\ 
& \le  \frac{1}{u^H \varphi(u)^H}\int_0^{\varphi(u)s} (ut)^{\alpha} (\varphi(u)s-v)^{\alpha} v^{-\gamma} {\mathrm d}v \non\\
& =  \frac{(ut)^\alpha (\varphi(u)s)^{\alpha -\gamma +1}} {u^H \varphi(u)^H} \int_0^1 (1-\theta)^{\alpha}\theta ^{-\gamma} {\mathrm d} \theta \non\\
&= (ts)^{H}\left( \frac{\,s\,}{\,t\,}  \cdot \frac{\varphi(u)}{u}\right)^{(1-\gamma)/2} \cdot \text{Beta} ( \alpha + 1 , 1 - \gamma)  \to 0 \qasq u \to 0+. \non
\end{align}
This corresponds to the first term of $\Psi(\cdot, \cdot) $ in (\ref{eqn-Psi}). Similarly, by $\alpha$-H\"older continuity of function $\,x \to x^{\alpha}\,$, $\,\alpha > 0 \,$, we have $(u t+v)^{\alpha} - v^{\alpha} \le (ut)^{\alpha}$ and hence, again by the change of variables, we have 
\begin{align} \label{eq: Thm6.1-3} 
&  \frac{1}{u^H \varphi(u)^H} \int_0^{\infty} ((u t+v)^{\alpha} - v^{\alpha}) ((\varphi(u) s+v)^{\alpha} - v^{\alpha})v^{-\gamma} {\mathrm d} v \\
&\le  \frac{1}{u^H \varphi(u)^H} \int^{\infty}_{0} (ut)^{\alpha} (\varphi(u) s)^{\alpha - \gamma + 1} ( ( 1 + \theta )^{\alpha} - \theta^{\alpha} ) \theta^{-\gamma}{\mathrm d} \theta \non \\
&= (ts)^{H}\left( \frac{\,s\,}{\,t\,}  \cdot \frac{\varphi(u)}{u}\right)^{(1-\gamma)/2} \cdot  \int^{\infty}_{0}  ( ( 1 + \theta )^{\alpha} - \theta^{\alpha} ) \theta^{-\gamma}{\mathrm d} \theta \to 0 \qasq u \to 0+. \non  
\end{align}
where the integral is finite as in \eqref{eq:integrability0}. This corresponds to the second term of $\Psi(\cdot, \cdot) $ in (\ref{eqn-Psi}).
Thus, we have shown that condition (iii), i.e.,  \eqref{eqn-LLIL-p-iii} holds with equality to zero. 

For (iv), we show that in probability, 
$$
\sup_{t \in (0,T]} \frac{|X(ut)|}{u^H (2 \log\log u^{-1})^{1/2}} \to 0 \quad \text{as}\quad  u \to 0+. 
$$
By the Fernique inequality (see, e.g., \cite{berman1985asymptotic}), we obtain for each $u$,  there exists 
$q>0$ such that 
\begin{align}
&P \left( \sup_{t \in (0,T]} \frac{|X(ut)|}{u^H (2 \log\log u^{-1})^{1/2}} > \ep \right) 
\non\\
& \le C  P\left(Z>\frac{\ep}{\sup_{t \in (0,T]} (u^H (2 \log\log u^{-1})^{1/2})^{-1}\Psi(ut,ut)^{1/2} + q} \right) \non\\
&= C  P\left(Z>\frac{\ep}{ (2 \log\log u^{-1})^{1/2})^{-1}\sup_{t \in (0,T]} ( \Psi(t,t)^{1/2} + q} \right) \non\\
& \to 0 \qasq u \to 0+, \non
\end{align}
 for all $\ep\ge \ep_0$ for some $\ep_0>0$ and some constant $C>0$, and $Z\sim N(0,1)$.  This proves that condition (iv) holds.

For (vi), we show that for each $\eta>0$, there exists a $\delta>0$ such that 
$$
\limsup_{\theta \to 1-} \sum_{n=1}^\infty \exp\left( \frac{-\eta (\theta^n)^{2H} \log n}{ \sup_{s,t \in[0,T], \Phi(s,t)\le \delta^2}  \Phi(\theta^n s,\theta^n t) }  \right) <\infty.
$$
Observe that, similar to \eqref{Cov-scale-identity},  for $r>0$,
$$
\Phi(rs,rt)= r^{2H} \Phi(s,t). 
$$
Thus, it becomes to show that for each $\eta>0$, there exists a $\delta>0$ such that 
$$
 \sum_{n=1}^\infty \exp\left( \frac{-\eta  \log n}{  \delta^2 }  \right)  = \sum_{n=1}^\infty  n^{-\eta/\delta^2} <\infty,
$$
which holds for any $\delta < \sqrt{\eta}$. 
This completes the proof. 
\end{proof}

\subsection{Compositions} \

We consider the composition $X(\lvert X(\cdot)\rvert) $ of $X(\cdot)$ itself. In Example 4.1 of \cite{arcones1995law}, by applying its Corollary 4.2, 
it is shown that for FBM $B^H$ with $1/2\le  H <1$,  given $b>0$, with probability one,
$$
\left\{\frac{B^H(|B^H(ub)|)}{ u^{H^2} (2 \log \log u^{-1})^{(H+1)/2}} \right\}
$$
is relatively compact as $u\to 0+$ and its limit set is $[-\sigma, \sigma]$ where 
$\sigma = b^{H^2} H^{H/2} (H+1)^{-(H+1)/2}$. 

We apply \cite[Corollary 4.2]{arcones1995law} to the GFBM $X$ in \eqref{def-X}.  We remark that the following result requires $\alpha>0$, since we need to use the $\alpha$-H{\"o}lder continuity of the function $t\to t^\alpha$ in the proof. 
Note that when $\gamma =0$, the condition $\alpha>0$ is equivalent to $H>1/2$ in the case of FBM $B^H$. 
However, for the GFBM $X$, as shown in Figure~\ref{fig-alpha-gamma},  in the region of $\alpha>0$ and $\gamma \in (0,1)$, the range of the Hurst parameter $H$ can cover the entire interval $H \in (0,1)$ (observing that when the $\alpha \approx 0$ and $\gamma \approx 1$, $H$ can be very close to 0). 

\begin{theorem} \label{thm-LLIL-composition}
If $\,\alpha > 0 \,$, with probability one, the set 
$$
\left\{\frac{X(|X(ub)|)}{ u^{H^2} (2 \log \log u^{-1})^{(H+1)/2}} , \, \, u > 0 \right\}
$$
is relatively compact, as $u\to 0+$, and its limit set is $[-\sigma, \sigma]$ where 
$$
\sigma = \sup_{0 \le r \le  \sqrt{\Psi(b,b)}}  \sqrt{\Psi(r,r)} (1-r^2/ \Psi(b,b))^{1/2}, \quad a.s.
$$
\end{theorem}

\begin{proof}
We verify the four conditions (i)-(iv) in \cite[Corollary 4.2]{arcones1995law}. 
Condition (i) requires $E[X(ut) X(us)] = u^{2H} E[X(t)X(s)]$, which holds by \eqref{Cov-scale-identity}. Condition (ii) requires that $\sup_{0 \le t \le T}|X(t)|<\infty$ a.s. 
Condition (iii) requires that
$$
\lim_{u\to1-} E[X(ut) X(t) ] = E\big[X(t)^2\big] \quad \text{for each} \quad t\ge 0.
$$
This follows from the continuity of $\Psi(s,t)$ in \eqref{eqn-Psi}. 

For Condition (iv), we show for each $s,t \ge 0$, 
$$
\lim_{u \to 0+} u^{-H} E[X(s) X(ut)] =0. 
$$
By \eqref{eqn-Psi}, for $u$ small enough such that $s>ut$, we have
\begin{align}
u^{-H} E[X(s) X(ut)] &= c^2 u^{-H}  \int_0^{ut} (s-v)^{\alpha} (ut-v)^{\alpha} v^{-\gamma} {\rm d}v  \non\\
& \qquad + c^2 u^{-H}  \int_0^{\infty} ((s+v)^{\alpha} - v^{\alpha}) ((ut+v)^{\alpha} - v^{\alpha})v^{-\gamma}{\rm d}v. \non
\end{align} 
For the first term, by change of variables, it is equal to 
\begin{align*}
c^2 u^{H} \int_0^{t} (s/u-v)^{\alpha} (t-v)^{\alpha} v^{-\gamma}{\rm d}v
&\le  c^2 u^{H} \int_0^{t} (s/u)^{\alpha} (t-v)^{\alpha} v^{-\gamma} {\rm d}v  \\
& =  c^2 u^{H-\alpha} \int_0^{t} s^{\alpha} (t-v)^{\alpha} v^{-\gamma} {\rm d}v \to 0 \qasq u \to 0+
\end{align*}
since $H- \alpha = 1/2- \gamma/2>0$. 
For the second term, 
\begin{align*}
 & u^{-H}  \int_0^{\infty} ((s+v)^{\alpha} - v^{\alpha}) ((ut+v)^{\alpha} - v^{\alpha})v^{-\gamma}{\rm d}v \\
 & \le u^{-H} \int_0^\infty s^\alpha ((ut+v)^{\alpha} - v^{\alpha})v^{-\gamma}{\rm d}v \\
 & = u^{-H} \int_0^\infty s^\alpha (ut)^{\alpha-\gamma+1}  ((1+\theta)^\alpha - \theta^\alpha) \theta^{-\gamma} {\rm d}\theta \\
 &= u^{(1-\gamma)/2} s^\alpha t^{\alpha-\gamma+1} \int_0^\infty ((1+\theta)^\alpha - \theta^\alpha) \theta^{-\gamma} {\rm d}\theta \\
 & \to 0 \qasq u \to 0+,
\end{align*}
where the integral is finite as in \eqref{eq:integrability0}. Thus we have verified Condition (iv). This completes the proof. 
\end{proof}

\section*{Acknowledgements}
Authors were thankful to the reviewers and associate editor for their careful reading and suggestions on  various improvements of the paper. In particular, the associate editor pointed out a gap in section 4 in the paper, which led to a significant improvement of our understanding of the process.  
Tomoyuki Ichiba was supported in part by NSF grants DMS-1615229 and DMS-2008427. 
Guodong Pang was supported in part by CMMI-1635410,  DMS/CMMI-1715875 and Army Research Office through grant W911NF-17-1-0019. 
Murad S. Taqqu was supported in part by a Simons Foundation grant 569118 at Boston University.


\bibliographystyle{plain}

\bibliography{GFBM}

\end{document}